\documentclass[12pt,a4paper]{amsart}
\usepackage[top=35mm, bottom=35mm, left=30mm, right=30mm]{geometry}
\usepackage{mathptmx}
\usepackage{mathrsfs}
\usepackage{amscd}
\usepackage{verbatim}
\usepackage{url}
\usepackage[all]{xy}
\usepackage{color}
\usepackage[colorlinks=true,citecolor=blue]{hyperref}
\usepackage{tikz}
\usetikzlibrary{arrows}
\usetikzlibrary{patterns}
\usetikzlibrary{lindenmayersystems,backgrounds}
\usepackage[colorinlistoftodos,prependcaption,textsize=tiny, textwidth=20mm,]{todonotes}

\theoremstyle{plain}

\newtheorem{thm}{Theorem}[section]
\newtheorem{lem}[thm]{Lemma}
\newtheorem{prop}[thm]{Proposition}

\theoremstyle{definition}
\newtheorem{defn}[thm]{Definition}

\begin{document}
\title{The structures of higher rank lattice actions on dendrites}

\author[E.~Shi]{Enhui Shi}
\address[E. Shi]{Soochow University, Suzhou, Jiangsu 215006, China}
\email{ehshi@suda.edu.cn}

\author[H.~Xu]{Hui Xu}
\address[H. Xu]{CAS Wu Wen-Tsun Key Laboratory of Mathematics, University of Science and
Technology of China, Hefei, Anhui 230026, China}
\email{huixu2734@ustc.edu.cn}



\maketitle

%%%%%%%%%%%%%%%%%%%%%%%%%%%%%%%%%%%%%%%%%%%%%%%%%%%%%%%%%%%%%%%%%%%%%%%%%%%%%%%%%%%%%%%%%%%%%%%%%%%%%%%%%%%%%%%%
%%%%%%%%%%%%%%%%%%%%%%%%%%%%%%%%%%%%%%%%%%%%%%%%%%%%%

\begin{abstract}
Let $\Gamma$ be a higher rank lattice acting on a nondegenerate dendrite $X$ with no infinite order points. We show that
there exists a nondegenerate subdendrite $Y$ which is $\Gamma$-invariant and satisfies the following items:

 (1) There is   an inverse system of finite actions $\{(Y_i, \Gamma):i=1,2,3,\cdots\}$ with monotone bonding maps $\phi_i: Y_{i+1}\rightarrow Y_i$ and with each $Y_i$ being a dendrite, such that $(Y, \Gamma|Y)$ is topologically conjugate to the inverse limit $(\underset{\longleftarrow}{\lim}(Y_i, \Gamma), \Gamma)$.

(2) The first point map $r:X\rightarrow Y$ is a factor map from $(X, \Gamma)$ to $(Y, \Gamma|Y)$; if $x\in X\setminus Y$, then $r(x)$ is an end point of $Y$ with infinite orbit; for each $y\in Y$, $r^{-1}(y)$ is contractible, that is there is a sequence $g_i\in \Gamma$ with ${\rm diam}(g_ir^{-1}(y))\rightarrow 0$.

\end{abstract}

\pagestyle{myheadings} \markboth{E. Shi and H. Xu} {Higher rank lattice actions }

\section{Introduction}

It was conjectured that  every continuous action on the circle by a higher rank lattice must
factor through a finite group action (the so called $1$-dimensional Zimmer's rigidity conjecture).  Burger-Monod \cite{Bu} and Ghys \cite{Gh99}
 proved independently the existence of finite orbits for higher rank lattice actions on the circle,
which translates the conjecture into the equivalent form: no higher rank lattice admits a total ordering which is invariant by left translations.
The latter was answered affirmatively by Witte-Morris for arithmetic groups of higher $\mathbb{Q}$-rank \cite{Wi2} and by Deroin-Hurtado
for general higher rank lattices \cite{DH}. Recently there has been a great progress on the Zimmer program for smooth higher
rank lattice actions on manifolds with dimensions $\geq 2$. We do not plan to list all the related results here and just suggest the readers to consult
\cite{Br, Fi} for the surveys.
\medskip

There has been intensively studied around group actions on dendrites very recently. One motivation  is that
dendrites can appeare as the limit sets of some Klein groups, the structures of which
are closely related to the geometric properties of $3$-dimensional hyperbolic manifolds (see e.g. \cite{Bo, Mi}). Also,
the compactifications of the Cayley graphs of free groups are dendrites, which are important for understanding the algebraic
properties of free groups.  Group actions on the circle have been
systematically investigated during the past few decades \cite{Gh01, Na1}. However, group actions on general curves
lack of the same depth of understanding. Dendrites and the circle lie on two
opposite ends of Peano curves in topologies. So, studying group actions on dendrites is the starting point for better understanding group actions
on curves or continua with higher dimensions.  Some people studied the Ghys-Margulis' alternative for dendrite homeomorphism groups (\cite{Shi, Mal, DM2}).
 One may consult \cite{AN, GM, SY} for the discussions around  the structures of minimal sets for group actions on dendrites. The algebraic structures of dendrite
homeomorphism groups were investigated in \cite{DM1}

\medskip
 Duchesne-Monod proved the existence of finite orbits for
higher rank lattice actions on dendrites (\cite{DM2}). This motivates us to consider further
 the structures of higher rank lattice actions on dendrites. Though there are examples indicating that the exact analogy
 to the rigidity results by Witte-Morris and Deroin-Hurtado mentioned above do not hold anymore for higher rank lattice actions on dendrites,
  the authors showed in \cite{SX} that such actions are very restrictive. The aim of the paper is to
 give a more detailed description of the structures of higher rank lattice actions on dendrites. Explicitly, we obtained
 the following theorem.

\begin{thm}\label{main theorem 1}
Let $\Gamma$ be a higher rank lattice acting on a nondegenerate dendrite $X$ with no infinite order points. Then
there exists a nondegenerate subdendrite $Y$ which is $\Gamma$-invariant and satisfies the following items:

(1) There is   an inverse system of finite actions $\{(Y_i, \Gamma):i=1,2,3,\cdots\}$ with monotone bonding maps $\phi_i: Y_{i+1}\rightarrow Y_i$ and with each $Y_i$ being a dendrite, such that $(Y, \Gamma|Y)$ is topologically conjugate to the inverse limit $(\underset{\longleftarrow}{\lim}(Y_i, \Gamma), \Gamma)$.

(2) The first point map $r:X\rightarrow Y$ is a factor map from $(X, \Gamma)$ to $(Y, \Gamma|Y)$; if $x\in X\setminus Y$, then $r(x)$ is an end point of $Y$ with infinite orbit; for each $y\in Y$, $r^{-1}(y)$ is contractible, that is there is a sequence $g_i\in \Gamma$ with ${\rm diam}(g_ir^{-1}(y))\rightarrow 0$..
\end{thm}

The proof of Theorem \ref{main theorem 1} relies on the non-left-orderability of higher rank lattices by Deroin-Hurtado, the elementarity
of higher rank lattice actions on dendrites by Duchesne-Monod, and a connection between semilinear-left-orderability of a group and its actions on dendrites
observed by the authors.
\medskip

We will not introduce the related notions and properties around group actions and dendrites. One may consult \cite{SX} for the details.

\section{General facts about semilinear preorders}

Recall that a binary relation $\preceq$ on a set $X$ is a {\it preorder} if it satisfies that
\medskip

\begin{itemize}
\item[(O1)] $x\preceq x$ for any $x\in X$;
\item[(O2)] if $x\preceq y$ and $y\preceq z$ then $x\preceq z$ for any $x,y,z\in X$.
\end{itemize}
\medskip

\noindent If, in addition, $\preceq$ satisfies that
\medskip

\begin{itemize}
\item[(O3)] if $x\preceq y$ and $y\preceq x$ then $x=y$ for any $x,y\in X$,
\end{itemize}
\medskip

\noindent then $\preceq$ is a {\it partial order} on $X$.
\medskip

A preorder (resp. partial order) on $X$ is called a {\it total preorder} ({resp. \it total  order}) on $X$ if
\medskip

\begin{itemize}
\item[(O4)] any $x,y\in X$ are comparable; that is either $x\preceq y$ or $y\preceq x$.
\end{itemize}
\medskip

Let $G$ be a group with $e$ being the unit. A preorder/ partial order/ total order $\preceq$ on $G$ is said to be {\it left-invariant} if
\medskip

\begin{itemize}
\item[(O5)] for any $x,y\in X$ and $g\in G$, $gx\preceq gy$ whenever $x\preceq y$.
\end{itemize}
\medskip

Finally, we say $\preceq$ is a {\it semilinear left preorder} (resp. {\it semilinear left partial order}) if it satisfies (O1), (O2), (O5)  (resp. (O1),(O2),(O3),(O5)) and
\medskip

\begin{itemize}
\item[(O6)] for any $x\in G$, any two elements of $\{ y\in G: y\preceq x\}$ are comparable;
\item[(O7)] for any $x,y\in G$, there is some $z\in G$ with $z\preceq x$ and $z\preceq y$.
\end{itemize}

\medskip

The following characterization of semilinear left preorder is similar to \cite[Theorem 1.1]{Kop}.
\begin{prop}\label{char of semi}
Let $\preceq$ be a semilinear left preorder on a group $G$. Then the positive cone $P:=\{g\in G: e\preceq g\}$ of $\preceq$ has the following properties:
\begin{itemize}
\item[(P1)]  $P$ is a semigroup;
\item[(P2)]  $P\cap P^{-1}$ is a subgroup of $G$;
\item[(P3)]  $G=P^{-1}\cdot P$;
\item[(P4)]  $P\cdot P^{-1}\subseteq P\cup P^{-1}$.
\end{itemize}
If $\preceq$ is further a  partial order, then
\begin{itemize}
\item[(P2')] $P\cap P^{-1}=\{e\}$;
\end{itemize}
is further a total preorder, then
\begin{itemize}
\item[(P5)] $G=P\cup P^{-1}$.
\end{itemize}
Conversely, a subset $P$ of $G$ satisfying (P1)-(P4) (resp. (P1)(P2')(P3)(P4)) determines a semilinear left preorder (resp. semilinear left partial order) on $G$.
\end{prop}
\begin{proof}$(\Longrightarrow)$
(P1) and (P2) are direct from the definition of  semilinear left preorder.
\medskip

To prove (P3), let $g\in G$. Take $u\in G$ with $u\preceq e$ and $u\preceq g$ by (O7). Then $g=u(u^{-1}g)$ and $e\preceq u^{-1}g$.
So, $G\subset P^{-1}\cdot P$. The inclusion $P^{-1}\cdot P\subset G$ is clear.
\medskip

To prove (P4), let $x, y\in P$. Then $x^{-1}\preceq e$ and $y^{-1}\preceq e$. From (O6), either $x^{-1}\preceq y^{-1}$ or
$y^{-1}\preceq x^{-1}$, which implies $xy^{-1}\in P\cup P^{-1}$.
\medskip

(P2') and (P5) are clear.
\medskip

$(\Longleftarrow)$ Define $x\preceq y$ if $x^{-1}y\in P$. We only check (O6) and (O7), the others are direct.
\medskip

For (O6), let $x, y, z\in G$ be such that $x\preceq z$ and $y\preceq z$. Then $x^{-1}z\in P$ and $z^{-1}y\in P^{-1}$.
By (P4), $x^{-1}y=x^{-1}z\cdot z^{-1}y\in P\cup P^{-1}$. This means either $x\preceq y$ or $y\preceq x$.
\medskip

For (O7), let $x, y\in G$. By (P3), we have $x^{-1}y=a\cdot b$, where $a\in P^{-1}$ and $b\in P$.
Then $a\preceq e$ and $a\preceq x^{-1}y$, which implies $xa\preceq x$ and $xa\preceq y$.

\end{proof}

According to Proposition \ref{char of semi}, we also say that a subsemigroup $P$ of $G$ satisfying (P1)-(P4) is a {\it semilinear left preorder} on $G$.

\begin{lem}\label{preorder}
Let $P$ be a semilinear left preorder on $G$.
\begin{itemize}
\item[(1)] For any $q\in P$, we have $q(P\cup P^{-1})q^{-1}\subset P\cup P^{-1}$.
\item[(2)] For any $g\in G$, we have $\{x\in G: x\preceq g\}=gP^{-1}$.
\item[(3)] Let $w$ be a nontrivial word composed by some elements of $P$. If $w$ is in $P\cap P^{-1}$ then every letter occurring in $w$ is in $P\cap P^{-1}$.
\end{itemize}
\end{lem}
\begin{proof}
(1) For any $x\in P$, we have $qxq^{-1}\in qPq^{-1}\subset P\cdot P^{-1}\subset P\cup P^{-1}$ by (P4). For any $x\in P^{-1}$, we have $xq^{-1}\in P^{-1}$.  Thus $qP^{-1}q^{-1}\subset P\cdot P^{-1}\subset P\cup P^{-1}$.

\medskip
(2) If $x\preceq g$ then $g^{-1}x\in P^{-1}$ and hence $x=g(g^{-1}x)\in g P^{-1}$.  For any $y\in P^{-1}$, we have $gy\preceq g$. Thus (2) holds.

\medskip
(3) Set $H=P\cap P^{-1}$ and $w=g_1\cdots g_n$ with $g_1,\cdots,g_n\in P$.
To the contrary, assume that $g_i\notin H$ for some $i\in\{1,\cdots,n\}$.  Then $g_i\succ e$ and
\begin{eqnarray*}
w&=&g_1\cdots g_n\succeq g_1\cdots g_{n-1}\succeq\cdots\succeq g_1\cdots g_{i}\\
&\succ& g_1\cdots g_{i-1}\succeq e.
\end{eqnarray*}
This contradicts that $w$ is in $H$. Hence each $g_i$ is  in $H$.
\end{proof}

\section{Left orderability  and semilinear left preorders}

\begin{defn}
Let $\preceq$ be a semilinear preorder on a set $X$. A subset $F\subset X$ is {\it coinitial} if for any $x\in X$ there is some $y\in F$ with $y\preceq x$.
\end{defn}

For a subset $S$ of a group $G$, we use ${\rm sgr}(S)$ to denote the semigroup generated by $S$.
It was shown in \cite[Theorem 2.7]{Kop} that a group admitting a semilinear left partial order is left-orderable. Now we generalize it to the case of preorder.

\begin{prop} \label{key}
Let $G$ be a group admitting a semilinear left preorder $\preceq$ and let $P$ be its positive cone. Set
$H=P\cap P^{-1}$ and $\widetilde{H}=\bigcup_{p\in P}\bigcap_{q^{-1}\preceq p^{-1}} q^{-1}Hq.$
 Then
\begin{itemize}
\item[(1)] $\widetilde{H}$ is a normal subgroup of $G$;
\item[(2)] if  $G$ is finitely generated and $H\neq G$, then $\widetilde{H}\neq G$;
\item[(3)] the quotient $G/\widetilde{H}$ is left-orderable.
\end{itemize}
\end{prop}
\begin{proof}
(1) From (P2), $H$ is a subgroup of $G$; thus $x\in \widetilde{H}$ implies $x^{-1}\in \widetilde{H}$ by the definition. Let $x,y\in  \widetilde{H}$. Suppose that $x\in \bigcap_{q^{-1}\preceq p_{1}^{-1}} q^{-1}Hq$ and $y\in \bigcap_{q^{-1}\preceq p_{2}^{-1}} q^{-1}Hq$ for some $p_1,p_2\in P$. Since $p_1^{-1}$ and $p_2^{-1}$ are comparable
by (O6), we may assume that $p_1^{-1}\preceq p_{2}^{-1}$. Thus $y\in \bigcap_{q^{-1}\preceq p_{2}^{-1}} q^{-1}Hq\subseteq \bigcap_{q^{-1}\preceq p_{1}^{-1}} q^{-1}Hq$ and hence $xy\in \bigcap_{q^{-1}\preceq p_{1}^{-1}} q^{-1}Hq\subset \widetilde{H}$. So $\widetilde{H}$ is a group.

 For each $p\in P$ and $g\in G$, from Lemma \ref{preorder} (2),
 \begin{eqnarray*}
 && g\left(\cap_{q^{-1}\preceq p^{-1}} q^{-1}Hq\right)g^{-1}=g\left(\cap_{x\in p^{-1}P^{-1}} xHx^{-1}\right)g^{-1}\\
 &&= \cap_{x\in p^{-1}P^{-1}} gxH(gx)^{-1}
 =\cap_{x\in gp^{-1}P^{-1}} xHx^{-1} \\
 && =  \cap_{x\preceq gp^{-1}} xHx^{-1} \subset \cap_{x\preceq y^{-1}} xHx^{-1}\\
 &&\subset \widetilde{H},
  \end{eqnarray*}
  where $y^{-1}$ is some element in $P^{-1}$ with $y^{-1}\preceq gp^{-1}$. Thus $\widetilde{H}$ is  normal in $G$.

\medskip
(2) Suppose that $G$ is finitely generated and $\{g_1,\cdots, g_n\}$ is a finite set of generators. To the contrary, assume that $G=\widetilde{H}$. Then, for each $i\in\{1,\cdots,n\}$, there is some $p_{i}\in P$ such that $g_i\in \cap_{q^{-1}\preceq p_{i}^{-1}}q^{-1}Hq$. Take some $p\in P$ with $p^{-1}\preceq p_{i}^{-1}$ for each $i\in\{1,\cdots,n\}$. Then we have
\[\{g_1,\cdots, g_n\}\subset \cap_{q^{-1}\preceq p^{-1}}q^{-1}Hq\subset p^{-1}Hp.\]
Thus $G\subset p^{-1}Hp$ and hence $G=H$.  So (2) holds.

\medskip
(3) We may assume that $H\neq G$; otherwise,  $\widetilde{H}=G$ and the conclusion is trivial.

\medskip
{\bf Claim 1.} For any finitely many $x_1,\cdots, x_n\in G\setminus\widetilde{H}$, there are some $\varepsilon_1,\cdots, \varepsilon_n\in\{-1,1\}$ such that
\begin{equation*}
 {\rm sgr}(x_1^{\varepsilon_1},\cdots,x_n^{\varepsilon_n})\cap \widetilde{H}=\emptyset.
\end{equation*}

\medskip
We show the Claim 1 by induction on $n$. Given $x\in G\setminus \widetilde{H}$,  suppose that there is a positive integer $k$ such that $x^{k}\in \widetilde{H}$. Then $x^{k}\in \bigcap_{q^{-1}\preceq p^{-1}} q^{-1}Hq$, by the definition of $\widetilde{H}$, for some $p\in P$ with $p^{-1}\preceq x$.  Let $q^{-1}\preceq p^{-1}$ be given.  Then $q^{-1}\preceq x$ and hence $qxq^{-1}\in P\cup P^{-1}$ by (P4). WLOG, we may assume that $qxq^{-1}\in P$.   Then $qx^{k}q^{-1}\in H$ implies that $qxq^{-1}\in H$, by Lemma \ref{preorder} (3). Thus $x\in \bigcap_{q^{-1}\preceq p^{-1}} q^{-1}Hq$ whence $x\in \widetilde{H}$. This contradicts the assumption.  So the Claim 1 holds for $n=1$.

Now assume that $n\geq 2$ and the Claim 1 holds for any $y_1,\cdots, y_m\in G\setminus\widetilde{H}$ with $m<n$.
By (O7), there is some $p\in P$ with
\begin{equation*}
p^{-1}\preceq x_1, \cdots, p^{-1}\preceq x_n.
\end{equation*}
Then $qx_1,\cdots, qx_n\in P$, for each $q^{-1}\preceq p^{-1}$. According to (P4) in Proposition \ref{char of semi}, we have $\{q x_1q^{-1}, \cdots, qx_nq^{-1}\}\subset P\cup P^{-1}$.
Thus  there are some $\vec{\varepsilon}(q)=(\varepsilon_1(q),\cdots, \varepsilon_n(q))\in\{-1,1\}^{n}$ such that
$\{q x_1^{\varepsilon_1(q)}q^{-1}, \cdots, qx_n^{\varepsilon_n(q)}q^{-1}\}\subset  P$.
For each $\vec{\varepsilon}=(\varepsilon_1,\cdots, \varepsilon_n)\in\{-1,1\}^{n}$, let
\[Q(\vec{\varepsilon})=\left\{q^{-1}\preceq p^{-1}: \{q x_1^{\varepsilon_1}q^{-1}, \cdots, qx_n^{\varepsilon_n}q^{-1}\}\subset  P\right\}.\]
Then $\{q^{-1}: q^{-1}\preceq p^{-1}\}=\bigcup_{\vec{\varepsilon}\in \{-1,1\}^{n}}Q(\vec{\varepsilon})$.

Let $P^{++}=\{g\in G: g\succ e\}$  and $P^{--}=\{g\in G: g\prec e\}$ be the strictly positive and negative cones respectively. Now we discuss into two cases.

\medskip
{\bf Case 1}. There is an $\vec{\varepsilon}=(\varepsilon_1,\cdots, \varepsilon_n)\in\{-1,1\}^{n}$ and a coinitial set $Q\subset Q(\varepsilon)$ such that for any $q^{-1}\in Q$,
\[\{q x_1^{\varepsilon_1}q^{-1}, \cdots, qx_n^{\varepsilon_n}q^{-1}\}\subset   P^{++}.\]
Noting that any word composed of $qx_1^{\varepsilon_1}q^{-1},\cdots, qx_n^{\varepsilon_n}q^{-1}$ will lie in $P^{++}$ for any $q^{-1}\in Q$ and  $Q$ is a coinitial set,  the semigroup $ {\rm sgr}(x_1^{\varepsilon_1},\cdots,x_n^{\varepsilon_n})$ has empty intersection with $\widetilde{H}$. Then the Claim holds.

\medskip
{\bf Case 2}.  For any $\vec{\varepsilon}=(\varepsilon_1,\cdots, \varepsilon_n)\in\{-1,1\}^{n}$, there is some $p^{-1}(\vec{\varepsilon})\preceq p^{-1}$ such that  for any $q^{-1}\in\{q^{-1}\in Q(\vec{\varepsilon}): q^{-1}\preceq p^{-1}(\vec{\varepsilon})\}$, we have
\[ \left\{i\in\{1,\cdots,n\}: qx_{i}^{\varepsilon_i}q^{-1}\in H\right\}\neq \emptyset.\]
Thus, for each $q^{-1}\in Q(\vec{\varepsilon})$, there is a partition $\{1,\cdots,n\}=A(q, \varepsilon)\cup B(q, \varepsilon) $ such that
\[ \{ qx_{i}^{\varepsilon_i}q^{-1}: i\in A(q, \varepsilon)\}\subset H\ \ \text{and }\ \ \{ qx_{i}^{\varepsilon_i}q^{-1}: i\in B(q, \varepsilon)\}\subset P^{++}.\]

Since $H\neq G$ by the assumption, we have $P^{++}=P\setminus H\neq \emptyset$ and hence $P^{++}$ is infinite.  Note that there are only finitely many partitions of $\{1,\cdots, n\}$. So there is an $\vec{\varepsilon}=(\varepsilon_1,\cdots, \varepsilon_n)\in\{-1,1\}^{n}$, a coinitial set $Q\subset Q(\vec{\varepsilon})$ and  a partition $\{1,\cdots,n\}=A \cup B $ with $A\neq\emptyset, B\neq\emptyset$ such that
\[ \{ qx_{i}^{\varepsilon_i}q^{-1}: i\in A \}\subset H\ \ \text{and }\ \ \{ qx_{i}^{\varepsilon_i}q^{-1}: i\in B \}\subset P^{++},\]
for each $q^{-1}\in Q$. WLOG, we may assume that $A=\{1,\cdots,k\}$ and $B=\{k+1,\cdots,n\}$ for some $k\in\{1,\cdots, n-1\}$.

Now by the induction hypothesis, there is some $\vec{\eta}=(\eta_1,\cdots,\eta_{k})\in\{-1,1\}^{k}$ such that
\[ {\rm sgr}(x_1^{\eta_1},\cdots,x_k^{\eta_k})\cap \widetilde{H}=\emptyset. \]
Then we conclude that $(\eta_1,\cdots,\eta_{k}, \varepsilon_{k+1},\cdots, \varepsilon_{n})\in\{-1,1\}^{n}$  satisfies
 \[ {\rm sgr}(x_1^{\eta_1},\cdots,x_k^{\eta_k}, x_{k+1}^{\varepsilon_{k+1}}, \cdots, x_{n}^{\varepsilon_{n}})\cap \widetilde{H}=\emptyset. \]
Indeed, let $w$ be a word  composed of $x_1^{\eta_1},\cdots,x_k^{\eta_k}, x_{k+1}^{\varepsilon_{k+1}}, \cdots, x_{n}^{\varepsilon_{n}}$.  If $x_{k+1}^{\varepsilon_{k+1}}, \cdots, x_{n}^{\varepsilon_{n}}$  do not occur in $w$, then the choice of $\vec{\eta}$ implies that $w$ is not in $\widetilde{H}$. If there are some letters of $x_{k+1}^{\varepsilon_{k+1}}, \cdots, x_{n}^{\varepsilon_{n}}$ occur in $w$, then for each $q^{-1}\in Q$, $qwq^{-1}\in P^{++}$ and hence $w$ is not in  $\widetilde{H}$ by the coinitiality of $Q$. Thus we complete the proof of the Claim 1.

\medskip
Now (3) is followed from some standard arguments (see \cite[Lemma 2.2.3]{Glass}). For convenience of the readers, we afford a detailed proof here.

\medskip
 By the principe of compactness, the Claim 1 implies the following directly.
\medskip

{\bf Claim 2}. There is a map $\varepsilon: G\setminus \widetilde{H}\rightarrow \{-1,1\}$ such that for any finite  $g_1,\cdots,g_n\subset G\setminus \widetilde{H}$, ${\rm sgr}\left(g_1^{\varepsilon(g_1)}, \cdots, g_{n}^{\varepsilon(g_n)}\right)\cap \widetilde{H}=\emptyset$.

\medskip
Let $\widetilde{P}=\{ g\in G\setminus\widetilde{H}: \varepsilon(g)=1\}$. Then it  is easy to verify that $\widetilde{P}$ is a subsemigroup of $G$ and $G=\widetilde{P}\cup \widetilde{H}\cup\widetilde{P}^{-1}$.
\medskip

{\bf Claim 3}.  $\widetilde{P}=\widetilde{H}\widetilde{P}\widetilde{H}$.
\medskip

 First $\widetilde{P}=e\widetilde{P}e\subset \widetilde{H}\widetilde{P}\widetilde{H}$. To show the converse, we conclude that for any $h\in \widetilde{H}$ and $q\in \widetilde{P}$,  $hq\in \widetilde{P}$. Otherwise,  $\varepsilon(hq)=-1$. Then $h^{-1}\in {\rm sgr}((hq)^{-1}, q)\cap \widetilde{H}$, which contradicts the choice of $\varepsilon$. Similarly, $qh\in \widetilde{P}$. Thus for any $h_1,h_2\in \widetilde{H}$ and $q\in \widetilde{P}$, $h_1qh_2\in \widetilde{P}$ and hence $\widetilde{H}\widetilde{P}\widetilde{H}\subset \widetilde{P}$. Thus Claim 3 holds.

\medskip
We define an order $\leq$ on the quotient group $G/\widetilde{H}=\{g\widetilde{H}: g\in G\}$ by
\[ f\widetilde{H}<g\widetilde{H}\ \ \ \text{ if and only if }\ \ \ f^{-1}g\in \widetilde{P}.\]
It is well defined by Claim 3 and is a total order by the equality $G=\widetilde{P}\cup \widetilde{H}\cup \widetilde{P}^{-1}$. It is obvious that $\leq$ is $G$-invariant, i.e. for any $g,x,y\in G$, $gx\widetilde{H}<gy\widetilde{H}$ whenever $x\widetilde{H}<y\widetilde{H}$.   Thus the quotient   $G/\widetilde{H}$ is left-orderable.

\end{proof}

\section{Existence of pointwise fixed arcs}

%\begin{lem}\cite[Corollary 6.10]{Rag}\label{Rag}
%Let $G$ be a connected Lie group which admits a finitely dimensional linear representation $\rho:  G\rightarrow GL(V)$ such that the kernel of $\rho$ is torsion-free. Let $\Gamma$ be a finitely generated subgroup of $G$. Then $G$ admits a torsion-free subgroup of finite index.
%\end{lem}
The following is the well-known Margulis' Normal Subgroup Theorem (see e.g. \cite[Theorem 8.1.2]{Zim}).
\begin{lem}\label{NST}
Let $G$ be a connected real semisimple Lie group with finite center and no compact factors, let   $\Gamma$ be an irreducible lattice of $G$.  Assume that $\mathbb{R}$-{\rm rank}$(G)\geq 2$.  Then every normal subgroup of $\Gamma$ either is contained in the center of $G$ and hence is finite or has finite index in $\Gamma$.
\end{lem}

The following is the Lemma 2.1 in \cite{SX}

\begin{lem} \label{fixed point}
Let $X$ a nondegenerate dendrite and $h:X\rightarrow X$ be a homeomorphism. If $h$ fixes an end point $e\in X$, then $h$ fixes another point $o\not=e$.
\end{lem}

The following is key to establish the structure theorem.

\begin{thm}\label{end fixed}
Let $\Gamma$ be a higher rank lattice acting on a nondegenerate dendrite $X$. If $\Gamma$ fixes some end point $z$ of $X$, then there is another point $s\in X$ such that $\Gamma$ fixes the arc $[z,s]$ pointwise.
\end{thm}

\begin{proof}
Now take a finite set $\{g_1,\cdots,g_n\}$ of generators of $\Gamma$. For each $i\in \{1,\cdots, n\}$, by Lemma \ref{fixed point}, we can take  $x_i\in X\setminus\{z\}$ fixed by $g_i$. Let $t$ be the  point such that $[z,t]=\cap_{i=1}^{n}[z,x_i]$. Define a positive cone $P$ of $\Gamma$ by
\[ P=\left\{g\in \Gamma: g^{-1}(t)\in [z,t]\right\}.\]
Then $P$ leads to a preorder $\preceq$ on $\Gamma$. Let $H=P\cap P^{-1}$ and $\widetilde{H}=\bigcup_{p\in P}\bigcap_{q^{-1}\preceq p^{-1}} q^{-1}Hq$. Note that $H$ is just the stabilizer of $t$ in $\Gamma$.

\medskip
{\bf Claim}. There is some point $s\in (z,t]$ fixed by $\Gamma$.
\medskip

We discuss into two cases for the proof of the Claim.

\medskip
{\bf Case 1}. There is a sequence $(f_i)$ in $P^{-1}$ such that $f_i(t)\rightarrow z$ as $i\rightarrow \infty$.
Then $P$ gives rise to a semilinear left preorder $\preceq$ on $\Gamma$: for $g, h\in\Gamma$,
$$g\preceq h\  \mbox{if and only if}\  [z, g(t)]\subset [z, h(t)].$$
By Proposition \ref{key}, $\Gamma$ admits a left-orderable quotient $\Gamma/\widetilde{H}$.  By Lemma \ref{NST}, either $\widetilde{H}$ is contained in the center of $G$ or $\Gamma/\widetilde{H}$ is a finite group, where $G$ is the ambient Lie group of $\Gamma$.  In the former case, $\Gamma/\widetilde{H}$ is also a higher rank lattice, which
contradicts Deroin-Hurtado's theorem in \cite{DH}.  In the latter case, $\Gamma/\widetilde{H}$ is a finite group. Then the orderability implies that it is trivial and hence $\widetilde{H}=\Gamma$. Then we have $H=\Gamma$, by Proposition \ref{key} (2). Thus $t$ is fixed by $\Gamma$ and take $s=t$.

\medskip
{\bf Case 2}. There is no sequence $(f_i)$ in $P^{-1}$ with $f_i(t)\rightarrow z$ as $i\rightarrow \infty$.

Fix a canonical ordering $<$ on $[z,t]$ with $z<t$. Let $s=\inf_{g^{-1}\in P^{-1}} g(t)$. Then $s\in (z,t]$ and we claim that $s$ is fixed by $\Gamma$. Indeed, let $g^{-1}\in P^{-1}$ be given. Suppose that $s=\lim\limits_{n\rightarrow\infty} \gamma_{n}^{-1}(t)$ for some sequence $(\gamma_n^{-1})$ in $P^{-1}$. Then $g^{-1}(s)\geq s$ with respect to the canonical ordering $<$ on $[z,t]$. If $g^{-1}(s)\neq s$ then $g^{-1}\gamma_n^{-1}(t)>\gamma_n^{-1}(t)$, for all sufficiently large $n$. Then $g\gamma_n^{-1}(t)<\gamma_n^{-1}(t)\leq t$ and hence $g\gamma_{n}^{-1}\in P^{-1}$. Thus
\[ s\leq \lim_{n\rightarrow\infty}g\gamma_{n}^{-1}(t)\leq \lim_{n\rightarrow\infty}\gamma_{n}^{-1}(t)=s.\]
So $g(s)=s$ and hence $g^{-1}(s)=s$. This contradiction shows that $s$ is fixed by each element in $P$. 
Since for each $i$, either $g_i\in P$ or $g_i^{-1}\in P$, $s$ is fixed by $\Gamma$.

\medskip
Thus the Claim holds. Now, by Denrion-Hurtado's theorem, $\Gamma$ fixes the $[z,s]$ pointwise.
\end{proof}

\section{Proof of the main theorem}

Now we are ready to prove the main theorem.

\begin{lem}\cite[Proposition 5.11]{DK}\label{index}
Let $G$ be a finitely generated group. For any $n\in\mathbb{N}$, there are finitely many subgroups of $G$ with indices less than $n$.
\end{lem}

\begin{proof}[Proof of the main theorem]  Along the same lines of the proof of Theorem 1.2 in \cite{SX},  Theorem \ref{main theorem 1} follows from  Theorem \ref{end fixed}
and Lemma \ref{index} (We need only use Theorem \ref{end fixed}
and Lemma \ref{index} instead of Proposition 5.3 and Lemma 5.5 in \cite{SX} respectively).
\end{proof}
%\begin{thm}
%Let $X$ be a nondegenerate dendrite without infinite order points and $\Gamma$ be a higher rank lattice. Then any action of $\Gamma$ on $X$ admits a $\Gamma$-invariant nondegenerate  subdendrite $Y$ such that the subsystem $(Y, \Gamma|Y)$ is almost finite  and the first point map $r: X\rightarrow Y$ is a contractible extension.
%\end{thm}

%\subsection*{Acknowledgements}

\end{document}